\documentclass[reqno]{amsart}
\usepackage{amsthm,amsmath,amssymb}
\usepackage{stmaryrd,mathrsfs}
\usepackage[T1]{fontenc}
\usepackage{esint}
\usepackage{tikz}

\allowdisplaybreaks

\newcommand{\ud}[0]{\,\mathrm{d}}

\newcommand{\abs}[1]{|#1|}

\newcommand{\Norm}[2]{\|#1\|_{#2}}

\newcommand{\ave}[1]{\langle #1\rangle}

\newcommand{\BMO}[0]{\operatorname{BMO}}

\renewcommand{\Re}[0]{\operatorname{Re}}

\newcommand{\R}{\mathbb{R}}
\newcommand{\C}{\mathbb{C}}

\newcommand{\eps}[0]{\varepsilon}

\swapnumbers \numberwithin{equation}{section}

\theoremstyle{plain}
\newtheorem{theorem}[equation]{Theorem}

\newtheorem{lemma}[equation]{Lemma}

\theoremstyle{definition}

\theoremstyle{remark}

\makeatletter
\@namedef{subjclassname@2010}{%
  \textup{2010} Mathematics Subject Classification}
\makeatother

\begin{document}

\title[Two-weight commutators]{The Holmes--Wick theorem on two-weight bounds for higher order commutators revisited}

\author{T.~P.\ Hyt\"onen}
\address{Department of Mathematics and Statistics, P.O.B.~68 (Gustaf H\"all\-str\"omin katu~2b), FI-00014 University of Helsinki, Finland}
\email{tuomas.hytonen@helsinki.fi}

%\date{\today}

\thanks{The author is supported by the ERC Starting Grant ``AnProb''. He is a member of the Finnish Centre of Excellence in Analysis and Dynamics Research.}
% 42B20 Singular and oscillatory integrals (Calder\'on-Zygmund, etc.) 
% 42B25 Maximal functions, Littlewood-Paley theory
% 42B35 Function spaces arising in harmonic analysis 

% 46B09 Probabilistic methods in Banach space theory
% 46E40 Spaces of vector- and operator-valued functions
% 47A60 Functional calculus
% 47F05 Partial differential operators
% 60G46 Martingales and classical analysis

\maketitle

\begin{abstract}
A sufficient condition for the two-weight boundedness of higher order commutators was recently obtained by Holmes and Wick in terms of an intersection of two BMO spaces. We provide an alternative proof, showing that the higher order case can be deduced by a classical Cauchy integral argument from the corresponding first order result of Holmes, Lacey and Wick.
\end{abstract}

\section{Introduction}

Let $T$ be a Calder\'on--Zygmund operator on $\R^n$, and $M_b$ be the pointwise multiplication operator $M_b:f\mapsto bf$ by a function $b$. A connection between the $L^p(\R^n)$-boundedness of the commutators
\begin{equation*}
  [b,T]:=M_b\circ T-T\circ M_b
\end{equation*}
and the bounded mean oscillation norm
\begin{equation*}
  \Norm{b}{BMO(\R^n)}
  :=\sup_Q\fint_Q\abs{b-\ave{b}_Q}\ud x
\end{equation*}
has been known since the seminal work of Coifman, Rochberg and Weiss \cite{CRW}. Here and below, $\sup_Q$ stands for the supremum over all cubes $Q\subset\R^n$, and
\begin{equation*}
   \ave{b}_Q:=\fint_Q b:=\frac{1}{\abs{Q}}\int_Q b\ud x
\end{equation*}
is the average of the function $b$ over the cube $Q$. It was proved in \cite{CRW} that the commutator $[b,T]$ is bounded on $L^p(\R^n)$ for $p\in(1,\infty)$ whenever $b\in BMO(\R^n)$ and $T$ is a Calder\'on--Zygmund operator. Moreover, when $T$ is a special Calder\'on--Zygmund operator like the Hilbert transform $H$ for $n=1$, or the vector $\vec{R}$ of the Riesz transforms for $n> 1$, this becomes ``if and only if'', showing that $BMO(\R^n)$ is precisely the correct function space for such commutator estimates.

It is remarkable that a similar characterisation is available in a much more general situation. Namely, consider two weight functions $\lambda,\mu$ in the Muckenhoupt class $A_p$, defined by the finiteness of the respective $A_p$ constants $[\lambda]_{A_p}$ and $[\mu]_{A_p}$, where
\begin{equation*}
  [w]_{A_p}:=\sup_Q\Big(\fint_Q w\Big)\Big(\fint_Q w^{1-p'}\Big)^{p-1},\qquad p\in(1,\infty).
\end{equation*}
Then it was shown by Bloom \cite{Bloom:85} that there is a bounded action of $[b,H]:L^p(\mu)\to L^p(\lambda)$ if and only if the weighted BMO norm
\begin{equation*}
  \Norm{b}{BMO(\nu)}:=\sup_Q\frac{1}{\nu(Q)}\int_Q\abs{b-\ave{b}_Q}\ud x
\end{equation*}
is finite, where the weight function $\nu:=(\mu/\lambda)^{1/p}$ is identified with the measure $\nu(E):=\int_E\nu\ud x$. Note that the norm of $BMO(\nu)$ still involves the unweighted average $\ave{b}_Q$ and an integral $\int_Q\abs{b-\ave{b}_Q}\ud x$ with respect to the Lebesgue measure, and the weight $\nu$ only makes an appearance in the normalisation by $1/\nu(Q)$.

Recently, Bloom's theorem was revisited by Holmes, Lacey and Wick, who gave a new proof of the original result \cite{HolLW:Bloom} and an extension to higher dimensions and general Calder\'on--Zygmund operators \cite{HolLW}. More precisely, they showed that the membership of $b$ in $BMO(\nu)$ still characterises the boundedness of $[b,T]:L^p(\mu)\to L^p(\lambda)$ when $T=\vec{R}$ is the vector of the Riesz transforms, and provides a sufficient condition for this boundedness for an arbitrary Calder\'on--Zygmund operator. We record the latter result in a quantitative form which, although not stated as such in \cite{HolLW}, follows readily by inspection of the same argument. (In fact, practically all known applications of the $A_p$ condition depend on upper bounds rather than exact values of the $A_p$ constants.)

\begin{theorem}[\cite{HolLW}]\label{thm:HolLW}
Let $T$ be a Calder\'on--Zygmund operator on $\R^n$, and $p\in(1,\infty)$. For any two weights $\lambda,\mu\in A_p$ and a function $b\in BMO(\nu)$, where $\nu=(\mu/\lambda)^{1/p}$, there holds
\begin{equation*}
  \Norm{[b,T]}{L^p(\mu)\to L^p(\lambda)}\leq C_{n,p,T}([\mu]_{A_p},[\lambda]_{A_p})\Norm{b}{BMO(\nu)},
\end{equation*}
where $C_{n,p,T}(\cdot,\cdot)$ is monotone increasing in both $A_p$ constants.
\end{theorem}

With Theorem \ref{thm:HolLW} at hand, the next natural object of study consist of the \emph{higher order commutators}
\begin{equation*}
  C^k_b(T):=[b,C^{k-1}_b(T)],\quad C^0_b(T):=T.
\end{equation*}
For this class of operators, a sufficient condition for the two-weight boundedness was provided by Holmes and Wick \cite{HolWick} in terms of the intersection of the classical and weighted BMO spaces:

\begin{theorem}[\cite{HolWick}]\label{thm:HolWick}
Let $T$ be a Calder\'on--Zygmund operator on $\R^n$, $p\in(1,\infty)$, and $k>1$. For any two weights $\lambda,\mu\in A_p$ and a function $b\in BMO\cap BMO(\nu)$, where $\nu=(\mu/\lambda)^{1/p}$, there holds
\begin{equation*}
  \Norm{C^k_b(T)}{L^p(\mu)\to L^p(\lambda)}\leq C_{n,p,k,T}([\mu]_{A_p},[\lambda]_{A_p})\Norm{b}{BMO}^{k-1}\Norm{b}{BMO(\nu)},
\end{equation*}
where $C_{n,p,k,T}(\cdot,\cdot)$ is monotone increasing in both arguments.
\end{theorem}

Both Theorems \ref{thm:HolLW} and \ref{thm:HolWick} were proved by modern methods of dyadic analysis, using the dyadic representation theorem from \cite{Hytonen:A2} to expand the Calder\'on--Zygmund operator $T$ in terms of simpler object called \emph{dyadic shifts} $S_{m,n}$, and exploiting their explicit structure to analyse each $C_b^k(S_{m,n})$. The goal of this paper is to provide an alternative approach to the higher order Theorem \ref{thm:HolWick}, based on a black-box application of the first order Theorem \ref{thm:HolLW}, combined with a Cauchy integral argument that goes back to the classical paper of Coifman, Rochberg and Weiss \cite{CRW}. This approach shows in particular that essentially all that we need to know about the operator $T$ to prove Theorem \ref{thm:HolWick} is encoded in the conclusions of Theorem \ref{thm:HolLW}; the deeper structural analysis of $T$ is only needed to establish this first order result. The careful reader will have noticed that we never gave a definition of a ``Calder\'on--Zygmund operator''; indeed, all we need to know is that it is a linear operator that satisfies the conclusions of Theorem \ref{thm:HolLW}!

\section{Preliminaries on weights}

Besides the $A_p$ constant defined above for $p\in(1,\infty)$, we shall need the Fujii--Wilson $A_\infty$ constant
\begin{equation*}
  [w]_{A_\infty}:=\sup_Q\frac{1}{w(Q)}\int_Q M(1_Q w),
\end{equation*}
where $M$ is the Hardy--Littlewood maximal operator. We shall quote several results from \cite{HytPer}, where the same quantity is denoted by $[w]_{A_\infty}'$ instead. It satisfies $[w]_{A_\infty}\leq c_n[w]_{A_p}$ for all $p\in(1,\infty)$, see \cite[bottom of p. 778]{HytPer}.

When $p\in(1,\infty)$ is fixed, we denote by $\sigma:=w^{1-p'}$ the dual weight, which satisfies $[\sigma]_{A_{p'}}=[w]_{A_p}^{p'-1}$ by simple algebra.
It is useful to define the quantity
\begin{equation*}
  (w)_{A_p}:=\max([w]_{A_\infty},[\sigma]_{A_\infty}),
%  \{w\}_{A_p}:=[w]_{A_p}^{\max(1,p'-1)}.
\end{equation*}
which satisfies
\begin{equation*}
  (w)_{A_p}
  \leq c_n\max([w]_{A_p},[\sigma]_{A_{p'}})
  = c_n\max([w]_{A_p},[w]_{A_{p}}^{p'-1})
  =c_n[w]_{A_p}^{\max(1,p'-1)}
\end{equation*}
by chaining the observations above.

We shall need the following relation of $A_p$ weights and the BMO space. This is certainly implicit in the literature and known to experts, but not easily citable in the stated form, so it included for completeness. The case $p=2$ can be found in \cite[Lemma 7.3]{HytPer}, and the argument here follows the same pattern.

\begin{lemma}\label{lem:ebw}
Let $p\in(1,\infty)$, $w\in A_p$, and $b\in BMO$ on $\R^n$. There are constants $\eps_{n,p},c_{n,p}>0$ depending only on the indicated parameters, such that
\begin{equation*}
  [e^{\Re(bz)}w]_{A_p}\leq c_{n,p}[w]_{A_p}
\end{equation*}
for all $z\in\C$ with
\begin{equation*}
  \abs{z}\leq\frac{\eps_{n,p}}{\Norm{b}{BMO}(w)_{A_p}}.
\end{equation*}
\end{lemma}

\begin{proof}
We recall that if $q\leq 1+\eps_n/[w]_{A_\infty}$, then $w$ satisfies the reverse H\"older inequality (cf. \cite[Theorem 2.3]{HytPer})
\begin{equation}\label{eq:RHI}
  \Big( \fint_Q w^q\Big)^{1/q}\leq 2\fint_Q w.
\end{equation}
Also, if $\Norm{b}{\BMO}\leq \eps_n$, then a version of the John--Nirenberg inequality says that
\begin{equation}\label{eq:JNI}
  \fint_Q e^{\abs{b-\ave{b}_Q}}\leq 2.
\end{equation}

Let $\sigma=w^{1-p'}$ be the dual weight and choose $q=1+\eps_n/(w)_{A_p}$. Then
%Thus
%\begin{equation*}
 % q'=1+c_n^{-1}\max([w]_{A_\infty},[\sigma]_{A_\infty})
%\end{equation*}
\begin{equation*}
\begin{split}
  &\Big(\fint_Q e^{\Re (bz)}w\Big)\Big(\fint_Q (e^{\Re (bz)}w)^{1-p'}\Big)^{p-1} \\
  &\leq \Big( \fint_Q w^q\Big)^{1/q} \Big( \fint_Q e^{q'\Re (bz)}\Big)^{1/q'}
    \Big( \fint_Q \sigma^q\Big)^{(p-1)/q} \Big( \fint_Q e^{q'\Re (bz)(1-p')}\Big)^{(p-1)/q'},
\end{split}
\end{equation*}
where, by the reverse H\"older inequality \eqref{eq:RHI} for both $w$ and $\sigma$,
\begin{equation*}
\begin{split}
  \Big( \fint_Q w^q\Big)^{1/q} \Big( \fint_Q \sigma^q\Big)^{(p-1)/q}
  \leq \Big(2\fint_Q w\Big)\Big( 2\fint_Q\sigma\Big)^{p-1}
  \leq 2^p[w]_{A_p}
\end{split}
\end{equation*}
and, multiplying and dividing by $e^{\ave{\Re(bz)}_Q}$,
\begin{equation*}
\begin{split}
  &\Big( \fint_Q e^{q'\Re (bz)}\Big)^{1/q'}\Big( \fint_Q e^{q'\Re (bz)(1-p')}\Big)^{(p-1)/q'} \\
  &= \Big( \fint_Q e^{q'(\Re (bz)-\ave{\Re (bz)}_Q)}\Big)^{1/q'}
      \Big( \fint_Q e^{q'(\Re (bz)-\ave{\Re (bz)}_Q)(1-p')}\Big)^{(p-1)/q'}  \\
   &   =:A^{1/q'}B^{(p-1)/q'}.
\end{split}
\end{equation*}
If $\abs{z}\leq \eps_n/(q'\Norm{b}{BMO})$, then $A\leq 2$, and if
\begin{equation*}
   \abs{z}\leq \eps_n/(q'\Norm{b}{BMO}(p'-1))=\eps_n(p-1)/(q'\Norm{b}{BMO}), 
\end{equation*}
then $B\leq 2$.
Thus, if $\abs{z}\leq \eps_n\min(1,p-1)/(q'\Norm{b}{BMO})$, then
\begin{equation*}
  A^{1/q'}B^{(p-1)/q'}\leq 2^{p/q'}\leq 2^p.
\end{equation*}
Altogether, recalling also the choice of $q=1+\eps_n/(w)_{A_p}$, so that $q'=1+\eps_n^{-1}(w)_{A_p}$, this shows that
\begin{equation*}
  \fint_Q e^{\Re (bz)}w\Big(\fint_Q (e^{\Re (bz)}w)^{1-p'}\Big)^{p-1}
  \leq 4^p[w]_{A_p}\quad\text{for }
  \abs{z}\leq\frac{\eps_n'\min(1,p-1)}{(w)_{A_p}\Norm{b}{BMO}},
\end{equation*}
and taking the supremum over $Q$ completes the proof.
\end{proof}

\section{New proof of Theorem \ref{thm:HolWick}}

\begin{proof}
For convenience, we write $k+1$ instead of $k$, so that $k\geq 1$.
Denoting
\begin{equation*}
   \tilde T:=C^1_b(T),\qquad
   F(z):=e^{bz}\tilde T e^{-bz},
\end{equation*}
we begin by observing (as in \cite[p. 621]{CRW}) that
\begin{equation*}
  C^{k+1}_b(T)
  =C^k_b(\tilde T)
  =F^{(k)}(0)
  =\frac{k !}{2\pi i}\oint \frac{F(z)\ud z}{z^{k+1}},
\end{equation*}
where the integral is over any closed path around the origin.
Thus
\begin{equation*}
\begin{split}
 & \Norm{C^k_b(\tilde T)}{L^p(\mu)\to L^p(\lambda)} \\
  &\leq\frac{k !}{2\pi }\oint_{\abs{z}=\delta} \Norm{e^{bz}\tilde T e^{-bz}}{L^p(\mu)\to L^p(\lambda)}\frac{\abs{\ud z}}{\abs{z}^{k+1}} \\
  &=\frac{k !}{2\pi }\oint_{\abs{z}=\delta} \Norm{\tilde T}{L^p(e^{\Re(bz)/p}\mu)\to L^p(e^{\Re(bz)/p}\lambda)}\frac{\abs{\ud z}}{\delta^{k+1}} \\
  &\leq \frac{k !}{2\pi }\oint_{\abs{z}=\delta}
    C_{n,p,T}([e^{\Re(bz)/p}\mu]_{A_p},[e^{\Re(bz)/p}\lambda]_{A_p}) \Norm{b}{BMO(\nu)}
    \frac{\abs{\ud z}}{\delta^{k+1}},
\end{split}
\end{equation*}
where we applied Theorem \ref{thm:HolLW} in the last step, observing that
\begin{equation*}
  \Big(\frac{e^{\Re(bz)/p}\mu}{e^{\Re(bz)/p}\lambda}\Big)^{1/p}
  =\Big(\frac{\mu}{\lambda}\Big)^{1/p}=\nu,
\end{equation*}
independently of $z$. 

By Lemma \ref{lem:ebw}, if
\begin{equation*}
  \delta=\frac{\eps_{n,p}}{\max\{(\mu)_{A_p},(\lambda)_{A_p}\}\Norm{b}{BMO}},
\end{equation*}
then
\begin{equation*}
  [e^{\Re(bz)/p}w]_{A_p}
  \leq c_{n,p} [w]_{A_p}\qquad w\in\{\mu,\lambda\}.
\end{equation*}
and the monotonicity of $C_{n,p,T}$ implies that
\begin{equation*}
\begin{split}
  C_{n,p,T}([e^{\Re(bz)/p}\mu]_{A_p},[e^{\Re(bz)/p}\lambda]_{A_p})
  &\leq C_{n,p,T}(c_{n,p}[\mu]_{A_p},c_{n,p}[\lambda]_{A_p}) \\
  &=:C_{n,p,T}'([\mu]_{A_p},[\lambda]_{A_p})
\end{split}
\end{equation*}

Substituting back, this gives
\begin{equation*}
\begin{split}
  & \Norm{C^k_b(T)}{L^p(\mu)\to L^p(\lambda)} \\
  &\leq \frac{k !}{2\pi }\oint_{\abs{z}=\delta}
    C_{n,p,T}'([\mu]_{A_p},[\lambda]_{A_p}) \Norm{b}{BMO(\nu)}
    \frac{\abs{\ud z}}{\delta^{k+1}} \\
 &=k !\cdot  C_{n,p,T}'([\mu]_{A_p},[\lambda]_{A_p}) \Norm{b}{BMO(\nu)}\frac{1}{\delta^k} \\
 &\leq C_{n,p,k,T}([\mu]_{A_p},[\lambda]_{A_p}) \Norm{b}{BMO(\nu)} \Norm{b}{BMO}^k,
\end{split}
\end{equation*}
where the chosen value of $\delta$ was substituted in the last step, hiding all admissible constants into the definition of $C_{n,p,k,T}$.
\end{proof}

The proof above shows a clear separation of the use of the two assumptions $b\in BMO(\nu)$ and $b\in BMO$ of Theorem \ref{thm:HolWick}: the former is only used for Theorem \ref{thm:HolLW} and the latter for bootstrapping this to the higher order case.

\bibliography{weighted}
\bibliographystyle{abbrv}

\end{document}